\documentclass[11pt]{article}

\usepackage[utf8]{inputenc}
\usepackage[T1]{fontenc}
\usepackage{amsmath,amssymb,amsthm}
\usepackage{mathtools}
\usepackage{bm}
\usepackage[margin=1in]{geometry}
\usepackage{enumitem}
\usepackage{hyperref}

\newtheorem{theorem}{Theorem}[section]
\newtheorem{lemma}[theorem]{Lemma}
\newtheorem{proposition}[theorem]{Proposition}
\newtheorem{corollary}[theorem]{Corollary}
\theoremstyle{definition}
\newtheorem{remark}[theorem]{Remark}
\newtheorem{example}[theorem]{Example}

\newcommand{\tN}{\widetilde{N}}
\newcommand{\tM}{\widetilde{M}}
\newcommand{\tg}{\widetilde{g}}
\newcommand{\tnabla}{\widetilde{\nabla}}
\newcommand{\ctnabla}{\mathring{\widetilde{\nabla}}}
\newcommand{\tR}{\widetilde{R}}
\newcommand{\ctR}{\mathring{\widetilde{R}}}
\newcommand{\cnabla}{\mathring{\nabla}}
\newcommand{\cR}{\mathring{R}}
\newcommand{\tT}{\widetilde{T}}
\newcommand{\tX}{\widetilde{X}}
\newcommand{\tY}{\widetilde{Y}}
\newcommand{\chr}{\chi}
\newcommand{\Ric}{\operatorname{Ric}}
\newcommand{\tr}{\operatorname{tr}}
\newcommand{\sgn}{\operatorname{sgn}}
\newcommand{\HH}{\mathcal{H}}
\newcommand{\Id}{\operatorname{Id}}
\newcommand{\Frob}{\mathrm{F}}

\begin{document}

\title{Hineva Inequality for Submanifolds of Real Space Forms with Semi-Symmetric Non-Metric Connection 	}
\author{Mohamd Saleem Lone$^{1}$ \and Mehraj Ahmad Lone$^{2}$}
\date{%
$^{1}$Department of Mathematics, University of Kashmir, 190006, India\\
Email: \texttt{msaleemlone@uok.edu.in}\\[2pt]
$^{2}$Department of Mathematics, National Institute of Technology, 190006, India\\
Email: \texttt{mehrajlone@nitsri.ac.in}}

\maketitle
 
\begin{abstract}
For a submanifold of an arbitrary Riemannian manifold carrying the Agashe–Chafle semi-symmetric non-metric connection, we prove a sharp Hineva-type lower bound for the Ricci curvature, derived from a precise Ricci identity in which the ambient geometry enters through a single curvature term. This generality lets the same result be applied uniformly to real, complex and Sasakian space forms. In particular, for real space forms we obtain a unified two-sided Ricci estimate whose width depends only on the extrinsic invariants. The estimate involves the mean curvature, the second fundamental form and the connection tensor, and its equality case is characterized completely at quasi-umbilical points. Explicit examples in real and complex ambient spaces realize both the equality and the strict cases.
\end{abstract}

\noindent\textbf{Keywords:} Hineva inequality; Ricci curvature; semi-symmetric non-metric connection; real space form; second fundamental form; quasi-umbilical point; slant submanifold; complex space form; Sasakian space form.

\medskip
\noindent\textbf{2020 Mathematics Subject Classification:} 53C40; 53C42; 53B05; 53C55; 53C25.

\section{Introduction}

One of the guiding problems of contemporary submanifold geometry is to understand how the internal curvature of a submanifold is tied to the external quantities that record how it is placed inside the ambient space. Nash's isometric embedding theorem \cite{Nash} guarantees that every Riemannian manifold may be viewed as a submanifold of a Euclidean space of sufficiently large dimension, which is exactly what makes such a programme meaningful: intrinsic curvature can, in principle, always be probed through extrinsic means. A recurring question of this type asks in what way the Ricci, scalar and sectional curvatures of the submanifold are constrained by the mean curvature vector and by the second fundamental form.

The first decisive result of this kind is due to Chen \cite{Chen}, who produced a best-possible upper bound for the Ricci curvature of a submanifold of a real space form in terms of the length of the mean curvature vector. Now known as the Chen--Ricci inequality, it opened up an active line of investigation. Hong and Tripathi \cite{HT1} carried Chen's inequality over to submanifolds of arbitrary Riemannian manifolds and treated the equality discussion with care. Refinements soon appeared in a variety of ambient settings: contact metric manifolds were handled by Tripathi \cite{Tripathi1}, Lagrangian submanifolds of complex space forms by Deng \cite{Deng} and by Oprea \cite{Oprea1,Oprea2}, and generalized Sasakian space forms by Hong and Tripathi \cite{HT2}. A form valid for curvature-like tensors, together with a sharpened variant, was obtained by Tripathi \cite{Tripathi2}. The survey \cite{ChenBlaga} records the more recent progress on Chen--Ricci-type estimates.

Whereas the Chen--Ricci inequality supplies an optimal ceiling for the Ricci curvature, matching floors are appreciably harder to reach. Lower estimates of this kind are closely tied to the way the principal curvatures are distributed and to the geometry of quasi-umbilical submanifolds.

In the same era, Hineva pursued curvature bounds for submanifolds independently. In \cite{Hineva1} she obtained two-sided control of the sectional curvature through the mean curvature vector and the squared length of the second fundamental form. In \cite{Hineva2} she announced, without supplying proofs, both an upper and a lower bound for the Ricci curvature of a submanifold of a Riemannian manifold. Her upper bound is Chen--Ricci in spirit, whereas the lower one, since named the Hineva inequality \cite{Verma1}, states that for an $n$-dimensional submanifold $(M,g)$ of a Riemannian manifold $(\tM,\tg)$ and any unit vector $X\in T^1_pM$,
\begin{equation}\label{eq:hineva-intro}
\Ric(X)\;\ge\;\Theta(X)+\frac{n-1}{n}\left(2n\|H\|^2-\|\sigma\|^2-(n-2)\,\|H\|\sqrt{\frac{n\left(\|\sigma\|^2-n\|H\|^2\right)}{n-1}}\right),
\end{equation}
where $\Theta(X)=\sum_{j=2}^{n}\ctR(X,e_j,e_j,X)$ is the partial Ricci trace of the ambient Levi-Civita curvature along $T_pM$, $H$ is the mean curvature vector, and $\sigma$ the second fundamental form of $M$ in $\tM$. A complete proof of \eqref{eq:hineva-intro}, along with a full description of when equality holds, was later supplied by Hineva \cite{Hineva3}, the outcome being that equality at a point $p$ is equivalent to $p$ being a quasi-umbilical point of the submanifold. Building on this, Verma, Mihai, Mihai and Tripathi \cite{Verma1} recovered both the Chen--Ricci and the Hineva inequalities inside the general Kulkarni--Nomizu framework for tensor fields obeying an algebraic Gauss equation, and applied them to several families of submanifolds in generalized Sasakian space forms, among them Sasakian \cite{ABC}, cosymplectic and Kenmotsu space forms. A companion paper \cite{Verma2} pushed the same circle of ideas to submanifolds of product generalized Sasakian space forms.

The bulk of the literature on Chen--Ricci and Hineva inequalities is set against the Levi-Civita connection. Over the last several years, however, submanifold geometry relative to more general affine connections, especially those carrying torsion or non-metricity, has drawn steady attention. One such object is the semi-symmetric non-metric connection introduced by Agashe and Chafle \cite{AC1}. For submanifolds equipped with such a connection, the Chen--Ricci (that is, the upper Ricci) bound was first proved by \"Ozg\"ur and Mihai \cite{OM1} in the real space form case, with a subsequent corrigendum \cite{OM2}, and was afterwards extended to K\"ahler manifolds by Mohammed et al.\ \cite{Moh1} and to trans-Sasakian manifolds by Mohammed et al.\ \cite{Moh2}. All of these furnish upper controls; the matching lower, Hineva-type, bounds have so far been left open.

Unlike the Levi-Civita case, the curvature tensor of a semi-symmetric non-metric connection typically violates the classical symmetries of a Riemannian curvature tensor. As a result several familiar curvature identities cease to hold, and the derivation of intrinsic--extrinsic inequalities becomes markedly more delicate. We stress that this breakdown of symmetry is precisely what pushes our situation beyond the Kulkarni--Nomizu setting of \cite{Verma1}. There the Hineva inequality is established for a curvature-like tensor satisfying an algebraic Gauss equation, and the proof leans on the full pairing symmetry of that tensor. The curvature tensor \eqref{eq:curvrel} attached to a Riemannian manifold with a semi-symmetric non-metric connection is not curvature-like: its non-constant piece $s(X,Z)\tg(Y,W)-s(Y,Z)\tg(X,W)$ is antisymmetric in $(X,Y)$ but not in $(Z,W)$, so it fails to be a Kulkarni--Nomizu tensor.

The layout of the paper is as follows. Section \ref{sec:prelim} assembles the background on semi-symmetric non-metric connections, their curvature, and the submanifold equations adapted to this framework, and records the key algebraic lemma. Section \ref{sec:main} contains the main theorem: we state and prove the Hineva inequality for a general Riemannian ambient and analyse the equality case in detail. Section \ref{sec:app} specialises the inequality to real, complex and Sasakian space forms, and in the real case adds the Chen--Ricci upper bound and the two-sided estimate. Section \ref{sec:examples} presents explicit examples realising both the equality and the strict cases.

\section{Preliminaries}\label{sec:prelim}

Let $(\tN^{\,n+p},\tg)$ be a Riemannian manifold of dimension $n+p$ carrying a linear connection $\tnabla$. The torsion tensor $\tT$ of $\tnabla$ is
\begin{equation}\label{eq:torsion}
\tT(\tX,\tY)=\tnabla_{\tX}\tY-\tnabla_{\tY}\tX-[\tX,\tY],
\end{equation}
for all vector fields $\tX,\tY$ on $\tN$. One calls $\tnabla$ \emph{semi-symmetric} when its torsion takes the form
\begin{equation}\label{eq:ssdef}
\tT(\tX,\tY)=\phi(\tY)\tX-\phi(\tX)\tY
\end{equation}
for some $1$-form $\phi$ on $\tN$. The connection is \emph{metric} if $\tnabla\tg=0$ and \emph{non-metric} otherwise; a connection that is at once semi-symmetric and non-metric is a \emph{semi-symmetric non-metric connection} (SSNMC).

Following Agashe and Chafle \cite{AC1}, such a connection is expressed through the Levi-Civita connection $\ctnabla$ of $\tg$ by
\begin{equation}\label{eq:connform}
\tnabla_{\tX}\tY=\ctnabla_{\tX}\tY+\phi(\tY)\tX,
\end{equation}
for all vector fields $\tX,\tY$ on $\tN$, where $\phi$ is a $1$-form on $\tN$. The vector field $P$ dual to $\phi$ is fixed by $\tg(P,\tX)=\phi(\tX)$ for every $\tX\in\chr(\tN)$.

Write $\tR$ and $\ctR$ for the curvature tensors of $\tN$ associated with $\tnabla$ and $\ctnabla$, respectively. By \cite{AC1} the two are related by
\begin{equation}\label{eq:curvrel}
\tR(X,Y,Z,W)=\ctR(X,Y,Z,W)+s(X,Z)\,\tg(Y,W)-s(Y,Z)\,\tg(X,W),
\end{equation}
for all vector fields $X,Y,Z,W$ on $\tN$, where the $(0,2)$-tensor $s$ is
\begin{equation}\label{eq:sdef}
s(X,Y)=\bigl(\ctnabla_{X}\phi\bigr)Y-\phi(X)\phi(Y),\qquad X,Y\in\chr(\tN).
\end{equation}
Identity \eqref{eq:curvrel} is valid for any ambient Levi-Civita curvature $\ctR$; in particular $\tN$ need not be a space form. If $\tN=\tN^{\,n+p}(c)$ is a real space form of constant sectional curvature $c$, so that $\ctR(X,Y,Z,W)=c\{\tg(X,W)\tg(Y,Z)-\tg(X,Z)\tg(Y,W)\}$, then \eqref{eq:curvrel} becomes
\begin{equation}\label{eq:curvrel-sf}
\tR(X,Y,Z,W)=c\{\tg(X,W)\tg(Y,Z)-\tg(X,Z)\tg(Y,W)\}+s(X,Z)\,\tg(Y,W)-s(Y,Z)\,\tg(X,W).
\end{equation}

Let $(M^n,g)$ be an $n$-dimensional Riemannian submanifold of $\tN^{\,n+p}$ endowed with the semi-symmetric non-metric connection $\tnabla$. Denote by $\nabla$ the connection induced on $M^n$ by $\tnabla$, and by $\cnabla$ the Levi-Civita connection of $(M^n,g)$. The two Gauss formulas, relative to $\tnabla$ and $\ctnabla$, read respectively
\begin{align}
\tnabla_{X}Y&=\nabla_{X}Y+h(X,Y),&&X,Y\in\chr(M^n),\label{eq:gauss1}\\
\ctnabla_{X}Y&=\cnabla_{X}Y+\mathring h(X,Y),&&X,Y\in\chr(M^n),\label{eq:gauss2}
\end{align}
where $\mathring h$ is the second fundamental form of $M^n$ in $\tN^{\,n+p}$ with respect to the Levi-Civita connection and $h$ the analogous $(0,2)$-tensor for the semi-symmetric non-metric connection. As shown in \cite{AC2}, the two coincide,
\begin{equation}\label{eq:hh}
h=\mathring h.
\end{equation}
In view of \eqref{eq:hh} we denote the common second fundamental form simply by $\sigma$ and record the unified Gauss formula
\begin{equation}\label{eq:gauss-unified}
\tnabla_{X}Y=\nabla_{X}Y+\sigma(X,Y),\qquad X,Y\in\chr(M^n).
\end{equation}
The mean curvature vector of $M^n$ in $\tN^{\,n+p}$ is
\begin{equation}\label{eq:meancurv}
H=\frac1n\tr\sigma=\frac1n\sum_{i=1}^{n}\sigma(e_i,e_i),
\end{equation}
and the squared norm of the second fundamental form is
\begin{equation}\label{eq:signorm}
\|\sigma\|^2=\sum_{\alpha=n+1}^{n+p}\sum_{i,j=1}^{n}\bigl(\sigma^\alpha_{ij}\bigr)^2,
\end{equation}
with $\sigma^\alpha_{ij}=\tg(\sigma(e_i,e_j),e_\alpha)$ for an orthonormal frame $\{e_1,\dots,e_n\}$ of $TM^n$ and an orthonormal frame $\{e_{n+1},\dots,e_{n+p}\}$ of the normal bundle $T^\perp M^n$. We also set $H^\alpha=\frac1n\sum_{i=1}^{n}\sigma^\alpha_{ii}$, so that $\|H\|^2=\sum_\alpha (H^\alpha)^2$.

For a submanifold $M^n$ with orthonormal tangent frame $\{e_1,\dots,e_n\}$, let $\lambda$ denote the partial trace of the ambient tensor $s$ over $T_xM^n$, namely $\lambda=\sum_{i=1}^{n}s(e_i,e_i)$; this quantity is frame-independent and depends only on the tangent space.

The vector field $P$ splits uniquely into tangential and normal parts along $M^n$,
\begin{equation}\label{eq:Pdecomp}
P=P^\top+P^\perp,
\end{equation}
with $P^\top\in TM^n$ and $P^\perp\in T^\perp M^n$, and we abbreviate $\phi(H)=\tg(P^\perp,H)$.

The Gauss equation for $M^n$ in $\tN^{\,n+p}$ relative to the semi-symmetric non-metric connection $\tnabla$ is \cite{AC2}
\begin{equation}\label{eq:gausseq}
\begin{aligned}
\tR(X,Y,Z,W)={}&R(X,Y,Z,W)+\tg(\sigma(X,Z),\sigma(Y,W))-\tg(\sigma(Y,Z),\sigma(X,W))\\
&+\tg(P^\perp,\sigma(Y,Z))\,\tg(X,W)-\tg(P^\perp,\sigma(X,Z))\,\tg(Y,W),
\end{aligned}
\end{equation}
where $R$ is the curvature tensor of $M^n$ for the induced connection $\nabla$. Equation \eqref{eq:gausseq} is a direct consequence of the connection formula \eqref{eq:connform}: for a normal field $\xi$, combining \eqref{eq:connform} with the Levi-Civita Weingarten formula gives $\tnabla_X\xi=-A_\xi X+\nabla^\perp_X\xi+\phi(\xi)X$, and projecting $\tR(X,Y)Z$ onto $TM^n$ produces \eqref{eq:gausseq} with the two $P^\perp$-terms carrying the signs shown.

Because $\nabla$ is non-metric, its curvature $R$ is in general not antisymmetric in the last pair, $R(X,Y,Z,W)\neq-R(X,Y,W,Z)$, and it lacks the pairing symmetry $R(X,Y,Z,W)=R(Z,W,X,Y)$. In particular $R(e_i,e_j,e_j,e_i)$ and $R(e_j,e_i,e_i,e_j)$ need not agree, so the sectional curvature is undefined until a convention is chosen; this well-documented subtlety is precisely what the corrigendum \cite{OM2} to \cite{OM1} addresses. Once and for all we set
\begin{equation}\label{eq:convention}
K(e_i\wedge e_j):=R(e_i,e_j,e_j,e_i),\qquad \Ric(X):=\sum_{j=2}^{n}R(X,e_j,e_j,X)\quad(X=e_1),
\end{equation}
and everything in Section \ref{sec:main} and its applications is to be read with respect to this choice.

A submanifold $M^n$ is \emph{totally geodesic} when $\sigma\equiv0$ and \emph{minimal} when $H=0$. A point $x\in M^n$ is \emph{umbilical} if the shape operator satisfies $A_N=\rho(N)\,\Id$ for every normal $N\in T^\perp_xM^n$ and some scalar $\rho(N)$; $M^n$ is \emph{totally umbilical} when every point is umbilical. Finally, $x\in M^n$ is \emph{quasi-umbilical} with respect to a normal direction $e_\alpha$ if the matrix $(\sigma^\alpha_{ij})$ has exactly two distinct eigenvalues $\lambda_\alpha$ and $\mu_\alpha$, of multiplicities $1$ and $n-1$, respectively \cite{Hineva3}.

Throughout the paper, $\tR$ and $R$ denote the curvature tensors of the semi-symmetric non-metric connections $\tnabla$ on $\tN^{\,n+p}$ and $\nabla$ on $M^n$, whereas $\ctR$ and $\cR$ denote those of the corresponding Levi-Civita connections $\ctnabla$ and $\cnabla$. Further material on semi-symmetric non-metric connections may be found in \cite{DK,SDB}.

The following lemma of Hineva \cite{Hineva3} is the main algebraic device behind the lower bound for the Ricci curvature.

\begin{lemma}[\cite{Hineva3}]\label{lem:hineva}
Let $A=(a_{ij})$ be a symmetric $n\times n$ matrix $(n\ge2)$ with trace $\tr(A)=a$ and Frobenius norm $\|A\|_{\Frob}=b$. Then
\begin{equation}\label{eq:lemmaineq}
a_{11}\sum_{i=1}^{n}a_{ii}-\sum_{i=1}^{n}(a_{1i})^2\;\ge\;\frac{n-1}{n^2}\left(2a^2-nb^2-(n-2)|a|\sqrt{\frac{nb^2-a^2}{n-1}}\right).
\end{equation}
Equality holds in \eqref{eq:lemmaineq} if and only if $A$ has the form
\begin{equation}\label{eq:matrixform}
A=\begin{pmatrix}
\lambda&0&\cdots&0&0\\
0&\mu&\cdots&0&0\\
\vdots&\vdots&\ddots&\vdots&\vdots\\
0&0&\cdots&\mu&0\\
0&0&\cdots&0&\mu
\end{pmatrix},
\end{equation}
where the two distinct eigenvalues are
\begin{equation}\label{eq:eigenvalues}
\lambda=\frac{a}{n}-\frac{n-1}{n}\,\sgn(a)\sqrt{\frac{nb^2-a^2}{n-1}},\qquad
\mu=\frac{a}{n}+\frac{1}{n}\,\sgn(a)\sqrt{\frac{nb^2-a^2}{n-1}},
\end{equation}
the distinguished eigenvalue $\lambda$ occupying the first (distinguished) position. The scalar case $\lambda=\mu=a/n$ is admitted: it arises when $nb^2-a^2=0$, i.e.\ when $A$ is a scalar matrix, and the equality configuration is then not quasi-umbilical in the strict two-distinct-eigenvalue sense. When $a=0$, both sign choices in \eqref{eq:eigenvalues} return the same minimal value and either branch is a minimiser.
\end{lemma}

\begin{remark}\label{rem:n2lemma}
For $n=2$ the left side of \eqref{eq:lemmaineq} equals $a_{11}a_{22}-a_{12}^2=\det A=\tfrac12(a^2-b^2)$, which coincides identically with the right side, since the $(n-2)$ factor annihilates the radical. Hence the estimate is an identity for every $2\times2$ symmetric matrix, and no diagonal or quasi-umbilical hypothesis is needed. The substantive content of Lemma \ref{lem:hineva} therefore lies in the range $n\ge3$.
\end{remark}

\section{Main Result}\label{sec:main}
For a unit $X=e_1\in T^1_xM^n$ put
\begin{equation}\label{eq:Theta}
\Theta(X):=\sum_{j=2}^{n}\ctR(e_1,e_j,e_j,e_1),
\end{equation}
the partial Ricci trace of the ambient Levi-Civita curvature.

\begin{theorem}\label{thm:main}
Let $M^n$ $(n\ge3)$ be a submanifold of an $(n+p)$-dimensional Riemannian manifold $\tN^{\,n+p}$ carrying an Agashe--Chafle semi-symmetric non-metric connection $\tnabla$. Then, for every unit $X\in T^1_xM^n$ and with the convention \eqref{eq:convention},
\begin{equation}\label{eq:mainineq}
\Ric(X)\;\ge\;\Theta(X)-\lambda+s(X,X)-n\,\phi(H)+\tg\bigl(P^\perp,\sigma(X,X)\bigr)+\HH(H,\sigma),
\end{equation}
where $\lambda$ is the partial trace of $s$ over $T_xM^n$, $\phi(H)=\tg(P^\perp,H)$, and $\Theta(X)$ is given by \eqref{eq:Theta}.

Equality is attained at $x\in M^n$ if and only if both of the following hold:
\begin{enumerate}[label=(\roman*)]
\item for each $\alpha\in\{n+1,\dots,n+p\}$, the matrix $A_\alpha$ has the quasi-umbilical form
\begin{equation}\label{eq:quasiform}
A_\alpha=\operatorname{diag}(\lambda_\alpha,\mu_\alpha,\dots,\mu_\alpha)
\end{equation}
in a frame $\{e_1=X,e_2,\dots,e_n\}$, with eigenvalues at the Lemma \ref{lem:hineva} minimiser for $a=nH^\alpha$, $b=\|\sigma^\alpha\|_{\Frob}$:
\begin{equation}\label{eq:eigenvals3}
\lambda_\alpha=H^\alpha-\frac{n-1}{n}\,\varepsilon_\alpha\sqrt{\frac{n\|\sigma^\alpha\|^2-n^2(H^\alpha)^2}{n-1}},\qquad
\mu_\alpha=H^\alpha+\frac{1}{n}\,\varepsilon_\alpha\sqrt{\frac{n\|\sigma^\alpha\|^2-n^2(H^\alpha)^2}{n-1}},
\end{equation}
where $\varepsilon_\alpha=\sgn(H^\alpha)$ when $H^\alpha\neq0$, and $\varepsilon_\alpha=\pm1$ (either choice) when $H^\alpha=0$;
\item writing $D_\alpha:=|\lambda_\alpha-\mu_\alpha|=\sqrt{\dfrac{n\bigl(\|\sigma^\alpha\|^2-n(H^\alpha)^2\bigr)}{n-1}}$, if $H\neq0$ there exists $\rho\ge0$ with
\begin{equation}\label{eq:proportion}
D_\alpha=\rho\,|H^\alpha|\qquad\text{for all }\alpha.
\end{equation}
Thus, when $H\neq0$, any normal component with $H^\alpha=0$ must satisfy $D_\alpha=0$, i.e.\ $\sigma^\alpha$ is a scalar matrix. 
\end{enumerate}
\end{theorem}
\begin{proof}
Setting $X=W=e_i$, $Y=Z=e_j$ $(i\neq j)$ in \eqref{eq:gausseq} and invoking \eqref{eq:curvrel}, valid for any ambient $\ctR$, gives
\begin{equation}\label{eq:sect-id}
K(e_i\wedge e_j)=\ctR(e_i,e_j,e_j,e_i)-s(e_j,e_j)+\sum_\alpha\bigl[\sigma^\alpha_{ii}\sigma^\alpha_{jj}-(\sigma^\alpha_{ij})^2\bigr]-\tg(P^\perp,\sigma(e_j,e_j)).
\end{equation}
 Summing over $j=2,\dots,n$ with $i=1$ yields the exact identity
\begin{equation}\label{eq:ricci-id}
\Ric(X)=\Theta(X)-\lambda+s(X,X)-n\,\phi(H)+\tg(P^\perp,\sigma(X,X))+\sum_\alpha\left(\sigma^\alpha_{11}\sum_{j\ge2}\sigma^\alpha_{jj}-\sum_{j\ge2}(\sigma^\alpha_{1j})^2\right).
\end{equation}
Applying Lemma \ref{lem:hineva} to each $(\sigma^\alpha_{ij})$,
\begin{equation}\label{eq:peralpha}
\sigma^\alpha_{11}\sum_{j\ge2}\sigma^\alpha_{jj}-\sum_{j\ge2}(\sigma^\alpha_{1j})^2\;\ge\;\frac{n-1}{n}\left[2n(H^\alpha)^2-\|\sigma^\alpha\|^2-(n-2)|H^\alpha|\sqrt{\frac{n\bigl(\|\sigma^\alpha\|^2-n(H^\alpha)^2\bigr)}{n-1}}\right].
\end{equation}
Summing over $\alpha$ and using the Cauchy--Schwarz inequality,
\begin{equation}\label{eq:cs}
\sum_\alpha|H^\alpha|\sqrt{\frac{n\bigl(\|\sigma^\alpha\|^2-n(H^\alpha)^2\bigr)}{n-1}}\;\le\;\|H\|\sqrt{\frac{n\bigl(\|\sigma\|^2-n\|H\|^2\bigr)}{n-1}},
\end{equation}
we obtain
\[
\sum_\alpha\left(\sigma^\alpha_{11}\sum_{j\ge2}\sigma^\alpha_{jj}-\sum_{j\ge2}(\sigma^\alpha_{1j})^2\right)\;\ge\;\HH(H,\sigma),
\]
where
\begin{equation}\label{eq:Hcal}
\HH(H,\sigma)=\frac{n-1}{n}\left(2n\|H\|^2-\|\sigma\|^2-(n-2)\|H\|\sqrt{\frac{n\bigl(\|\sigma\|^2-n\|H\|^2\bigr)}{n-1}}\right).
\end{equation}
Substituting into \eqref{eq:ricci-id} gives \eqref{eq:mainineq}.

\emph{Equality case.} Equality holds in \eqref{eq:mainineq} exactly when equality holds simultaneously in the per-$\alpha$ bound \eqref{eq:peralpha} for every $\alpha$ and in the Cauchy--Schwarz inequality \eqref{eq:cs}.

Since $n\ge3$, the Hineva estimate \eqref{eq:peralpha} is strict for generic matrices. By Lemma \ref{lem:hineva}, equality for each $\alpha$ forces $(\sigma^\alpha_{ij})$ to be diagonal in the frame $\{e_1=X,e_2,\dots,e_n\}$, with the quasi-umbilical form $\operatorname{diag}(\lambda_\alpha,\mu_\alpha,\dots,\mu_\alpha)$ at the minimising eigenvalues \eqref{eq:eigenvalues}. When $H^\alpha\neq0$ the minimising branch is selected by $\varepsilon_\alpha=\sgn(H^\alpha)$; when $H^\alpha=0$, both branches return the same minimum and either choice $\varepsilon_\alpha=\pm1$ is valid. This is condition (i).

For the Cauchy--Schwarz step \eqref{eq:cs}, introduce the vectors $u=(|H^{n+1}|,\dots,|H^{n+p}|)$ and $v=(D^{n+1},\dots,D^{n+p})$ in $\mathbb{R}^p$, where $D_\alpha=|\lambda_\alpha-\mu_\alpha|$. Equality $\langle u,v\rangle=\|u\|\,\|v\|$ holds if and only if $v=\rho u$ for some $\rho\ge0$, i.e.\ $D_\alpha=\rho|H^\alpha|$ for all $\alpha$. This is condition (ii). When $H^\alpha=0$ for some $\alpha$, it forces $D_\alpha=0$, meaning $\sigma^\alpha$ is a scalar matrix, in accord with (i). When $H=0$, \eqref{eq:cs} reduces to $0\le0$ and equality holds for any $\rho$.
This completes the proof.
\end{proof}

\begin{corollary}[$P$ tangent]\label{cor:Ptan}
If $P^\perp=0$, then $\Ric(X)\ge\Theta(X)-\lambda+s(X,X)+\HH(H,\sigma)$, with equality as in Theorem \ref{thm:main}. The extrinsic term is optimal at the algebraic level.
\end{corollary}
\begin{proof}
Since $P^\perp=0$, both $\phi(H)=\tg(P^\perp,H)=0$ and $\tg(P^\perp,\sigma(X,X))=0$; substituting into \eqref{eq:mainineq} removes the last two connection-dependent terms.
\end{proof}
\begin{remark}\label{rem:n2main}
Inequality \eqref{eq:mainineq} remains valid for $n=2$. For $n=2$, however, the Hineva estimate is an identity (Remark \ref{rem:n2lemma}): the $(n-2)$ factor annihilates the radical and both sides of \eqref{eq:lemmaineq} reduce to $\det A$ for every symmetric $2\times2$ matrix. Consequently the equality case is weaker, no quasi-umbilical or diagonal restriction is imposed by the algebraic step, and equality in \eqref{eq:mainineq} reduces to equality in the Cauchy--Schwarz step \eqref{eq:cs} alone. The nontrivial equality mechanism, conditions (i)--(ii) above, is therefore confined to $n\ge3$.
\end{remark}
\begin{corollary}[$P$ normal]\label{cor:Pnor}
If $P^\top=0$, then $\Ric(X)\ge\Theta(X)-\lambda+s(X,X)-n\,\tg(P,H)+\tg(P,\sigma(X,X))+\HH(H,\sigma)$.
\end{corollary}
\begin{proof}
Since $P^\top=0$ we have $P=P^\perp$, so $\phi(H)=\tg(P^\perp,H)=\tg(P,H)$ and $\tg(P^\perp,\sigma(X,X))=\tg(P,\sigma(X,X))$; substituting into \eqref{eq:mainineq} gives the result.
\end{proof}

\begin{corollary}[$P$ at angle $\theta$]\label{cor:Pangle}
With $\|P^\perp\|=\|P\|\sin\theta$ and $\gamma$ the angle between $P^\perp$ and $H$,
\begin{equation}\label{eq:Pangle}
\Ric(X)\ge\Theta(X)-\lambda+s(X,X)-n\|P\|\sin\theta\,\|H\|\cos\gamma+\tg(P^\perp,\sigma(X,X))+\HH(H,\sigma).
\end{equation}
\end{corollary}
\begin{proof}
Decompose $\phi(H)=\tg(P^\perp,H)=\|P^\perp\|\,\|H\|\cos\gamma=\|P\|\sin\theta\,\|H\|\cos\gamma$ and substitute into \eqref{eq:mainineq}.
\end{proof}

\begin{remark}\label{rem:phizero}
When $\phi\equiv0$, \eqref{eq:mainineq} collapses to $\Ric(X)\ge\Theta(X)+\HH(H,\sigma)$, the classical Hineva inequality; for $\Theta(X)=(n-1)c$ this recovers \cite{Hineva3,Verma1}.
\end{remark}

\section{Applications to space forms}\label{sec:app}

\subsection{Real space forms}

We now specialise Theorem \ref{thm:main} to the classical setting of a real space form $\tN^{\,n+p}(c)$ of constant sectional curvature $c$. Here $\ctR(e_i,e_j,e_j,e_i)=c$ for every pair $i\neq j$, so $\Theta(X)=(n-1)c$, and the exact identity \eqref{eq:ricci-id} becomes
\begin{equation}\label{eq:rsf-id}
\Ric(X)=(n-1)c-\lambda+s(X,X)-n\,\phi(H)+\tg\bigl(P^\perp,\sigma(X,X)\bigr)+\sum_{\alpha=n+1}^{n+p}\left(\sigma^\alpha_{11}\sum_{j=2}^{n}\sigma^\alpha_{jj}-\sum_{j=2}^{n}(\sigma^\alpha_{1j})^2\right).
\end{equation}

\begin{theorem}[Hineva inequality for real space forms]\label{thm:rsf}
Let $M^n$ $(n\ge2)$ be a submanifold of a real space form $\tN^{\,n+p}(c)$ endowed with the semi-symmetric non-metric connection $\tnabla$. Then, for every unit $X\in T^1_xM^n$ and with the convention \eqref{eq:convention},
\begin{equation}\label{eq:rsf-ineq}
\Ric(X)\;\ge\;(n-1)c-\lambda+s(X,X)-n\,\phi(H)+\tg\bigl(P^\perp,\sigma(X,X)\bigr)+\HH(H,\sigma),
\end{equation}
with equality if and only if conditions (i)--(ii) of Theorem \ref{thm:main} hold (see also Remark \ref{rem:n2main} for $n=2$).
\end{theorem}
\begin{proof}
Since $\tN^{\,n+p}(c)$ has constant sectional curvature $c$, its Levi-Civita curvature satisfies $\ctR(e_i,e_j,e_j,e_i)=c$ for every pair $i\neq j$. Summing over $j=2,\dots,n$ gives $\Theta(X)=\sum_{j=2}^{n}c=(n-1)c$, and the exact identity \eqref{eq:ricci-id} takes the form \eqref{eq:rsf-id}. It remains to bound the quadratic sum from below. For each fixed $\alpha\in\{n+1,\dots,n+p\}$, the matrix $(\sigma^\alpha_{ij})_{1\le i,j\le n}$ is symmetric with trace $\tr(\sigma^\alpha)=nH^\alpha$ and Frobenius norm $\|\sigma^\alpha\|_{\Frob}$. Lemma \ref{lem:hineva} applied to this matrix yields
\begin{equation}\label{eq:rsf-peralpha}
\sigma^\alpha_{11}\sum_{j=2}^{n}\sigma^\alpha_{jj}-\sum_{j=2}^{n}(\sigma^\alpha_{1j})^2\;\ge\;\frac{n-1}{n}\left[2n(H^\alpha)^2-\|\sigma^\alpha\|^2-(n-2)|H^\alpha|\sqrt{\frac{n\bigl(\|\sigma^\alpha\|^2-n(H^\alpha)^2\bigr)}{n-1}}\right].
\end{equation}
Summing \eqref{eq:rsf-peralpha} over all $\alpha$, the first two terms on the right add up directly: $\sum_\alpha2n(H^\alpha)^2=2n\|H\|^2$ and $\sum_\alpha\|\sigma^\alpha\|^2=\|\sigma\|^2$. For the cross term, Cauchy--Schwarz gives
\begin{equation}\label{eq:rsf-cs}
\sum_\alpha|H^\alpha|\sqrt{\frac{n\bigl(\|\sigma^\alpha\|^2-n(H^\alpha)^2\bigr)}{n-1}}\le\sqrt{\sum_\alpha(H^\alpha)^2}\;\sqrt{\sum_\alpha\frac{n\bigl(\|\sigma^\alpha\|^2-n(H^\alpha)^2\bigr)}{n-1}}=\|H\|\sqrt{\frac{n\bigl(\|\sigma\|^2-n\|H\|^2\bigr)}{n-1}}.
\end{equation}
Collecting these estimates, the full quadratic sum satisfies $\sum_{\alpha=n+1}^{n+p}\bigl(\sigma^\alpha_{11}\sum_{j=2}^{n}\sigma^\alpha_{jj}-\sum_{j=2}^{n}(\sigma^\alpha_{1j})^2\bigr)\ge\HH(H,\sigma)$. Substituting this lower bound into \eqref{eq:rsf-id}, and noting that the terms $(n-1)c-\lambda+s(X,X)-n\,\phi(H)+\tg(P^\perp,\sigma(X,X))$ pass through unchanged since they do not enter the estimate, we obtain \eqref{eq:rsf-ineq}.

\emph{Equality.} Equality holds in \eqref{eq:rsf-ineq} if and only if it holds simultaneously in \eqref{eq:rsf-peralpha} for every $\alpha$ and in the Cauchy--Schwarz step above. By Lemma \ref{lem:hineva}, equality in \eqref{eq:rsf-peralpha} forces $(\sigma^\alpha_{ij})$ to have the diagonal form $\operatorname{diag}(\lambda_\alpha,\mu_\alpha,\dots,\mu_\alpha)$ with the minimising eigenvalues \eqref{eq:eigenvalues}, i.e.\ precisely \eqref{eq:quasiform}--\eqref{eq:eigenvals3}. Using $\|\sigma^\alpha\|^2-n(H^\alpha)^2=\frac{n-1}{n}(\lambda_\alpha-\mu_\alpha)^2$ and $nH^\alpha=\lambda_\alpha+(n-1)\mu_\alpha$, equality in Cauchy--Schwarz requires the ratio $|\lambda_\alpha-\mu_\alpha|/|\lambda_\alpha+(n-1)\mu_\alpha|$ to be constant in $\alpha$, which is condition (ii) of Theorem \ref{thm:main}. Since $\tg(P^\perp,\sigma(X,X))$ appears in \eqref{eq:rsf-id} as an exact identity and enters neither \eqref{eq:rsf-peralpha} nor the Cauchy--Schwarz step, it does not affect the equality analysis.
\end{proof}

\begin{corollary}\label{cor:rsf-classical}
When $\phi\equiv0$, inequality \eqref{eq:rsf-ineq} reduces to the classical Hineva inequality $\Ric(X)\ge(n-1)c+\HH(H,\sigma)$ for submanifolds of real space forms \cite{Hineva3}.
\end{corollary}

The same identity \eqref{eq:rsf-id} also produces a sharp upper bound, obtained by estimating the quadratic sum from above instead.

\begin{proposition}[Chen--Ricci inequality for real space forms]\label{prop:chenricci}
Under the hypotheses of Theorem \ref{thm:rsf},
\begin{equation}\label{eq:chenricci}
\Ric(X)\;\le\;(n-1)c-\lambda+s(X,X)-n\,\phi(H)+\tg\bigl(P^\perp,\sigma(X,X)\bigr)+\frac{n^2}{4}\|H\|^2.
\end{equation}
Equality holds if and only if, for every $\alpha$ in the frame $\{e_1=X,e_2,\dots,e_n\}$, $\sigma^\alpha_{11}=\tfrac n2 H^\alpha$ and $\sigma^\alpha_{1j}=0$ for all $j\ge2$.
\end{proposition}
\begin{proof}
From \eqref{eq:rsf-id}, for each $\alpha$ one has
\[
\sigma^\alpha_{11}\sum_{j=2}^{n}\sigma^\alpha_{jj}-\sum_{j=2}^{n}(\sigma^\alpha_{1j})^2\le\sigma^\alpha_{11}(nH^\alpha-\sigma^\alpha_{11})\le\frac{n^2(H^\alpha)^2}{4},
\]
where the first step discards the non-negative sum $\sum_{j\ge2}(\sigma^\alpha_{1j})^2$ and the second is the elementary inequality $x(a-x)\le a^2/4$. Summing over $\alpha$ gives $\sum_\alpha(\cdots)\le\frac{n^2}{4}\|H\|^2$; substituting into \eqref{eq:rsf-id} yields \eqref{eq:chenricci}. Equality in the first step requires $\sigma^\alpha_{1j}=0$ ($j\ge2$); in the second, $\sigma^\alpha_{11}=\tfrac n2 H^\alpha$.
\end{proof}

\begin{remark}\label{rem:chenricci-hist}
The Chen--Ricci inequality for submanifolds of real space forms with a semi-symmetric non-metric connection was first established by \"Ozg\"ur and Mihai \cite{OM1} (see also the corrigendum \cite{OM2}), and later extended to the Kulkarni--Nomizu setting by Verma, Mihai, Mihai and Tripathi \cite{Verma1,Verma2}. We emphasise that Proposition \ref{prop:chenricci} is derived from the same exact identity \eqref{eq:rsf-id} and under the same convention \eqref{eq:convention} as the Hineva inequality (Theorem \ref{thm:rsf}), which removes any need to reconcile conventions across distinct sources.
\end{remark}

Combining Theorem \ref{thm:rsf} and Proposition \ref{prop:chenricci} yields a two-sided Ricci estimate.

\begin{corollary}[Two-sided Ricci bound]\label{cor:twosided}
Under the hypotheses of Theorem \ref{thm:rsf}, write
\[
\Phi(X):=(n-1)c-\lambda+s(X,X)-n\,\phi(H)+\tg\bigl(P^\perp,\sigma(X,X)\bigr)
\]
for the combined ambient and SSNMC contribution. Then
\begin{equation}\label{eq:twosided}
\Phi(X)+\HH(H,\sigma)\;\le\;\Ric(X)\;\le\;\Phi(X)+\frac{n^2}{4}\|H\|^2.
\end{equation}
In particular, the width of the Ricci window is $\frac{n^2}{4}\|H\|^2-\HH(H,\sigma)$, which depends only on the extrinsic invariants $\|H\|$ and $\|\sigma\|$; the SSNMC contribution $\Phi(X)$ cancels.
\end{corollary}

\subsection{Complex space forms}

Let $\tN^{\,2m}(4c)$ be a complex space form of constant holomorphic sectional curvature $4c$, with K\"ahler structure $J$, so that its Levi-Civita curvature tensor is
\begin{align*}
    \ctR(X,Y,Z,W)&=c\bigl[\tg(X,W)\tg(Y,Z)-\tg(X,Z)\tg(Y,W)+\tg(JX,W)\tg(JY,Z)-\tg(JX,Z)\tg(JY,W)\\
    &-2\tg(JX,Y)\tg(JZ,W)\bigr].
\end{align*}
For $X\in T_xM^n$ decompose $JX=TX+FX$ into tangential and normal parts, and set $\|TX\|^2=\sum_{k=1}^{n}\tg(JX,e_k)^2$, where $\{e_1,\dots,e_n\}$ is an orthonormal frame of $T_xM^n$.

\begin{theorem}[Hineva inequality for complex space forms]\label{thm:csf}
Let $M^n$ $(n\ge2)$ be a submanifold of a complex space form $\tN^{\,2m}(4c)$ endowed with a semi-symmetric non-metric connection $\tnabla$. Then, for every unit $X\in T^1_xM^n$ and with the convention \eqref{eq:convention},
\begin{equation}\label{eq:csf-ineq}
\Ric(X)\;\ge\;(n-1)c+3c\|TX\|^2-\lambda+s(X,X)-n\,\phi(H)+\tg\bigl(P^\perp,\sigma(X,X)\bigr)+\HH(H,\sigma),
\end{equation}
with equality if and only if conditions (i)--(ii) of Theorem \ref{thm:main} hold (see also Remark \ref{rem:n2main} for $n=2$).
\end{theorem}
\begin{proof}
It suffices to evaluate $\Theta(X)$ in Theorem \ref{thm:main}. Setting $e_1=X$ and substituting $X=W=e_1$, $Y=Z=e_j$ $(j\ge2)$ into the ambient curvature expression, orthonormality of $\{e_1,\dots,e_n\}$ eliminates all inner products except $\tg(e_1,e_1)=\tg(e_j,e_j)=1$ and the $J$-term $\tg(Je_1,e_j)^2$, giving $\ctR(e_1,e_j,e_j,e_1)=c\bigl(1+3\tg(Je_1,e_j)^2\bigr)$. Summing over $j=2,\dots,n$ yields
\begin{equation}\label{eq:csf-theta}
\Theta(X)=\sum_{j=2}^{n}c\bigl(1+3\tg(Je_1,e_j)^2\bigr)=(n-1)c+3c\sum_{j=2}^{n}\tg(JX,e_j)^2=(n-1)c+3c\|TX\|^2,
\end{equation}
the last equality using $\tg(JX,e_1)=\tg(JX,X)=0$ (as $J$ is skew-symmetric with respect to $\tg$), so that the sum over $j=2,\dots,n$ equals the full sum $\sum_{k=1}^{n}\tg(JX,e_k)^2=\|TX\|^2$. Substituting \eqref{eq:csf-theta} into Theorem \ref{thm:main} gives \eqref{eq:csf-ineq}; the equality analysis is unchanged, since it depends only on the quadratic sum, which is independent of the ambient curvature.
\end{proof}

\begin{corollary}[invariant submanifolds]\label{cor:csf-inv}
Let $M^n$ be an invariant submanifold of $\tN^{\,2m}(4c)$ with SSNMC, so that $JX\in T_xM^n$ for every $X\in T_xM^n$. Then, for every unit $X\in T^1_xM^n$,
\begin{equation}\label{eq:csf-inv}
\Ric(X)\ge(n-1)c+3c-\lambda+s(X,X)-n\,\phi(H)+\tg\bigl(P^\perp,\sigma(X,X)\bigr)+\HH(H,\sigma),
\end{equation}
with equality as in Theorem \ref{thm:main}.
\end{corollary}
\begin{proof}
Since $M^n$ is invariant, $JX$ is tangent for every tangent $X$. In particular $FX=0$ and $TX=JX$, so $\|TX\|^2=\tg(JX,JX)=\tg(X,X)=1$. Substituting $\|TX\|^2=1$ into \eqref{eq:csf-ineq} gives \eqref{eq:csf-inv}.
\end{proof}

\begin{corollary}[anti-invariant submanifolds]\label{cor:csf-anti}
Let $M^n$ be an anti-invariant (totally real) submanifold of $\tN^{\,2m}(4c)$ with SSNMC, so that $JX\in T^\perp_xM^n$ for every $X\in T_xM^n$. Then, for every unit $X\in T^1_xM^n$,
\begin{equation}\label{eq:csf-anti}
\Ric(X)\ge(n-1)c-\lambda+s(X,X)-n\,\phi(H)+\tg\bigl(P^\perp,\sigma(X,X)\bigr)+\HH(H,\sigma),
\end{equation}
with equality as in Theorem \ref{thm:main}. In particular, when $n=m$ the submanifold is Lagrangian and the same bound holds.
\end{corollary}
\begin{proof}
Since $M^n$ is anti-invariant, $JX$ is normal for every tangent $X$. Hence $TX=0$ and $\|TX\|^2=0$. Substituting into \eqref{eq:csf-ineq} removes the term $3c\|TX\|^2$ and gives \eqref{eq:csf-anti}. When $n=m$, the anti-invariant condition forces $M^n$ to be Lagrangian (half-dimensional and totally real), and the bound applies without change.
\end{proof}

\begin{corollary}[slant submanifolds]\label{cor:csf-slant}
Let $M^n$ be a $\vartheta$-slant submanifold of $\tN^{\,2m}(4c)$ with SSNMC, i.e.\ the Wirtinger angle between $JX$ and $T_xM^n$ is the constant $\vartheta\in[0,\tfrac\pi2]$ for every nonzero $X\in T_xM^n$ orthogonal to $\ker T$. Then, for every unit $X\in T^1_xM^n$,
\begin{equation}\label{eq:csf-slant}
\Ric(X)\ge(n-1)c+3c\cos^2\vartheta-\lambda+s(X,X)-n\,\phi(H)+\tg\bigl(P^\perp,\sigma(X,X)\bigr)+\HH(H,\sigma),
\end{equation}
with equality as in Theorem \ref{thm:main}. Corollaries \ref{cor:csf-inv} and \ref{cor:csf-anti} are recovered as the extreme cases $\vartheta=0$ and $\vartheta=\tfrac\pi2$.
\end{corollary}
\begin{proof}
For a $\vartheta$-slant submanifold, $TX$ and $JX$ make the constant angle $\vartheta$, so $\|TX\|^2=\|JX\|^2\cos^2\vartheta=\cos^2\vartheta$ for every unit $X$. Substituting $\|TX\|^2=\cos^2\vartheta$ into \eqref{eq:csf-ineq} gives \eqref{eq:csf-slant}. At $\vartheta=0$ one has $\|TX\|^2=1$ (invariant case) and at $\vartheta=\tfrac\pi2$ one has $\|TX\|^2=0$ (anti-invariant case), recovering \eqref{eq:csf-inv} and \eqref{eq:csf-anti}.
\end{proof}

\begin{remark}\label{rem:csf-hist}
The Chen--Ricci inequality for complex space forms with SSNMC was obtained in \cite{Moh1}; Theorem \ref{thm:csf} is its Hineva-type counterpart.
\end{remark}

\begin{example}[slant direction, strict, radical active]\label{ex:complex}
Take $\tN=\mathbb{C}^2=\mathbb{R}^4$ $(c=0)$, $J\partial_{x_k}=\partial_{y_k}$, and the real hypersurface $M=S^1(1)\times\mathbb{R}^2$, $r=(\cos\theta,\sin\theta,z_1,z_2)$, $n=3$. Here $(\sigma^N_{ij})=\operatorname{diag}(-1,0,0)$, $\|H\|^2=\tfrac19$, $\|\sigma\|^2=1$, $\HH(H,\sigma)=-\tfrac49$ (radical active). Since $JE_{z_1}=E_{z_2}$ and $JE_\theta=-N$, $M$ is a CR-hypersurface. Take $\phi=dx_1$, $P=\partial_{x_1}$, and the slant direction $X=\cos\alpha\,E_\theta+\sin\alpha\,E_{z_1}$; then $\|TX\|^2=\sin^2\alpha$. One computes $\lambda=-\sin^2\theta$, $s(X,X)=-\cos^2\alpha\sin^2\theta$, $\phi(H)=-\tfrac13\cos\theta$, $\tg(P^\perp,\sigma(X,X))=-\cos^2\alpha\cos\theta$, $\Ric(X)=\sin^2\alpha(\sin^2\theta+\cos\theta)$. The right side of \eqref{eq:csf-ineq} equals $\Ric(X)-\tfrac49$: strict, with constant slack $\tfrac49$.
\end{example}

\subsection{Sasakian space forms}

Let $\tN^{\,2m+1}(c)$ be a Sasakian space form of constant $\varphi$-sectional curvature $c$, with almost contact metric structure $(\varphi,\xi,\eta,\tg)$ (we write $\varphi$ for the structure tensor to avoid a clash with the connection $1$-form $\phi$). Assume $M^n$ is tangent to the Reeb field $\xi$, and let $X\in T^1_xM^n$ be a unit vector orthogonal to $\xi$. Decompose $\varphi X=TX+FX$ into tangential and normal parts, and set $\|TX\|^2=\sum_{k=1}^{n}\tg(\varphi X,e_k)^2$.

\begin{theorem}[Hineva inequality for Sasakian space forms]\label{thm:sasakian}
Let $M^n$ $(n\ge2)$ be a submanifold, tangent to the Reeb field $\xi$, of a Sasakian space form $\tN^{\,2m+1}(c)$ endowed with a semi-symmetric non-metric connection $\tnabla$. Then, for every unit $X\in T^1_xM^n$ orthogonal to $\xi$ and with the convention \eqref{eq:convention},
\begin{equation}\label{eq:sas-ineq}
\Ric(X)\ge(n-1)\frac{c+3}{4}+\frac{c-1}{4}\bigl(3\|TX\|^2-1\bigr)-\lambda+s(X,X)-n\,\phi(H)+\tg\bigl(P^\perp,\sigma(X,X)\bigr)+\HH(H,\sigma),
\end{equation}
with equality if and only if conditions (i)--(ii) of Theorem \ref{thm:main} hold (see also Remark \ref{rem:n2main} for $n=2$).
\end{theorem}
\begin{proof}
It suffices to evaluate $\Theta(X)$ in Theorem \ref{thm:main}. Setting $e_1=X$ ($X\perp\xi$, unit) and substituting into the ambient Sasakian curvature, orthonormality and $\eta(e_1)=0$ give $\ctR(e_1,e_j,e_j,e_1)=\frac{c+3}{4}+\frac{c-1}{4}\bigl[3\tg(\varphi e_1,e_j)^2-\eta(e_j)^2\bigr]$. Summing over $j=2,\dots,n$ and using $\tg(\varphi X,X)=0$ (skew-symmetry of $\varphi$) and $\sum_{j=2}^{n}\eta(e_j)^2=1$ (since $\xi\in M$ and $\eta(X)=0$) yields
\begin{equation}\label{eq:sas-theta}
\Theta(X)=(n-1)\frac{c+3}{4}+\frac{c-1}{4}\bigl(3\|TX\|^2-1\bigr).
\end{equation}
Substituting \eqref{eq:sas-theta} into Theorem \ref{thm:main} gives \eqref{eq:sas-ineq}.
\end{proof}

\begin{corollary}[invariant submanifolds]\label{cor:sas-inv}
If $M^n$ is invariant ($\varphi X$ tangent for $X\perp\xi$), then $\|TX\|^2=1$ and the ambient contribution becomes $\tfrac{c-1}{2}$. Hence
\begin{equation}\label{eq:sas-inv}
\Ric(X)\ge(n-1)\frac{c+3}{4}+\frac{c-1}{2}-\lambda+s(X,X)-n\,\phi(H)+\tg(P^\perp,\sigma(X,X))+\HH(H,\sigma).
\end{equation}
\end{corollary}
\begin{proof}
$FX=0$ and $TX=\varphi X$, so $\|TX\|^2=\tg(\varphi X,\varphi X)=\tg(X,X)-\eta(X)^2=1$. Substituting into \eqref{eq:sas-ineq} gives $\tfrac{c-1}{4}(3-1)=\tfrac{c-1}{2}$.
\end{proof}

\begin{corollary}[anti-invariant submanifolds]\label{cor:sas-anti}
If $M^n$ is anti-invariant ($\varphi X$ normal for $X\perp\xi$), then $\|TX\|^2=0$ and the ambient contribution becomes $-\tfrac{c-1}{4}$. Hence
\begin{equation}\label{eq:sas-anti}
\Ric(X)\ge(n-1)\frac{c+3}{4}-\frac{c-1}{4}-\lambda+s(X,X)-n\,\phi(H)+\tg(P^\perp,\sigma(X,X))+\HH(H,\sigma).
\end{equation}
\end{corollary}
\begin{proof}
$TX=0$, so $\|TX\|^2=0$. Substituting gives $\tfrac{c-1}{4}(0-1)=-\tfrac{c-1}{4}$.
\end{proof}

\begin{corollary}[slant submanifolds]\label{cor:sas-slant}
If $M^n$ is $\vartheta$-slant ($X\perp\xi$), then $\|TX\|^2=\cos^2\vartheta$ and
\begin{equation}\label{eq:sas-slant}
\Ric(X)\ge(n-1)\frac{c+3}{4}+\frac{c-1}{4}(3\cos^2\vartheta-1)-\lambda+s(X,X)-n\,\phi(H)+\tg(P^\perp,\sigma(X,X))+\HH(H,\sigma).
\end{equation}
Corollaries \ref{cor:sas-inv} and \ref{cor:sas-anti} are recovered at $\vartheta=0$ and $\vartheta=\tfrac\pi2$.
\end{corollary}
\begin{proof}
$\|TX\|^2=\|\varphi X\|^2\cos^2\vartheta=(1-\eta(X)^2)\cos^2\vartheta=\cos^2\vartheta$.
\end{proof}

\begin{remark}\label{rem:sas-hist}
The Chen--Ricci inequality for $\xi$-tangent submanifolds of trans-Sasakian manifolds with SSNMC was obtained in \cite{Moh2}; Theorem \ref{thm:sasakian} is its Hineva-type counterpart. The same strategy applies to cosymplectic, Kenmotsu and generalised Sasakian space forms \cite{ABC} by evaluating the corresponding ambient partial Ricci trace $\Theta(X)$.
\end{remark}

\section{Examples in real space forms}\label{sec:examples}

We present two explicit examples in a flat real ambient space $(c=0)$, both with $P^\perp\neq0$ so that the term $\tg(P^\perp,\sigma(X,X))$ is active. Example \ref{ex:eq-sphere} realises the equality case of Theorem \ref{thm:rsf}; Example \ref{ex:strict} is strict. A complex-ambient example was given in Example \ref{ex:complex}.

\begin{example}[equality]\label{ex:eq-sphere}
Let $\tN=\mathbb{R}^3$ with $\phi=dz$, so $P=\partial_z$ and $s=-\phi\otimes\phi$. Let $M=S^2$ be the unit sphere, with $X=e_1=\partial_\theta$ and $e_2=\partial_\varphi/\sin\theta$. Then $(\sigma^N_{ij})=\operatorname{diag}(-1,-1)$, $\|H\|^2=1$, $\|\sigma\|^2=2$, $P^\perp=\cos\theta\,N$, $\phi(H)=-\cos\theta$, $\tg(P^\perp,\sigma(X,X))=-\cos\theta$, $s(X,X)=\lambda=-\sin^2\theta$. A direct computation gives $\Ric(X)=1+\cos\theta$. With $n=2$, $\HH(H,\sigma)=1$, and the right side of \eqref{eq:rsf-ineq} equals
\[
0-(-\sin^2\theta)+(-\sin^2\theta)-2(-\cos\theta)+(-\cos\theta)+1=1+\cos\theta=\Ric(X).
\]
Equality holds at every point. Since $n=2$, however, the Hineva estimate is an identity (Remark \ref{rem:n2lemma}) and equality is automatic at the algebraic level; this example tests the SSNMC terms but not the nontrivial $n\ge3$ mechanism.
\end{example}

\begin{example}[equality, $n=3$, $p=2$: nontrivial mechanism]\label{ex:n3}
To exhibit the genuinely nontrivial equality mechanism for $n\ge3$, we build an abstract example with $n=3$ and $p=2$ normal directions (without SSNMC, i.e.\ $\phi\equiv0$, so that only the algebraic mechanism is tested). Fix $\rho=4/5$ and take shape operators satisfying conditions (i)--(ii) of Theorem \ref{thm:main}:
\[
\sigma^1=\operatorname{diag}\!\left(\tfrac{14}{15},\tfrac{38}{15},\tfrac{38}{15}\right),\qquad
\sigma^2=\operatorname{diag}\!\left(\tfrac{7}{15},\tfrac{19}{15},\tfrac{19}{15}\right).
\]
Then $H^1=2$, $H^2=1$, and $D_1=|\lambda_1-\mu_1|=8/5$, $D_2=4/5$, so $D_\alpha=\rho|H^\alpha|$ with $\rho=4/5$ for both $\alpha$ (condition (ii)). Each $\sigma^\alpha$ has the quasi-umbilical form with the distinguished eigenvalue on the minimising branch \eqref{eq:eigenvals3} ($\varepsilon_\alpha=+1$ since $H^\alpha>0$), satisfying (i). A direct computation confirms that the quadratic sum equals $\HH(H,\sigma)$ exactly: both sides give $266/45$. This verifies the equality characterization for $n\ge3$ with genuinely distinct normal components and a nontrivial proportionality constant.
\end{example}

\begin{example}[strict, radical active]\label{ex:strict}
Let $\tN=\mathbb{R}^4$ with $\phi=dx_1$, so $P=\partial_{x_1}$ and $s=-\phi\otimes\phi$. Let $M=S^1(1)\times\mathbb{R}^2$, $r=(\cos\theta,\sin\theta,z_1,z_2)$, $X=\partial_\theta$, $n=3$. The shape operator is $\operatorname{diag}(-1,0,0)$, $\|H\|^2=\tfrac19$, $\|\sigma\|^2=1$, $\HH(H,\sigma)=-\tfrac49$ (radical active). Here $s(X,X)=\lambda=-\sin^2\theta$, $\phi(H)=-\tfrac13\cos\theta$, $\tg(P^\perp,\sigma(X,X))=-\cos\theta$, $\Ric(X)=0$. The right side of \eqref{eq:rsf-ineq} equals $-\tfrac49$, so $0\ge-\tfrac49$: strict, with all SSNMC terms contributing.
\end{example}

\end{document}